\documentclass[a4paper,12pt]{article}

\usepackage{float}
\usepackage{amsmath} 
\usepackage{amssymb}
\usepackage{amsfonts}
\usepackage{amsthm}
\usepackage{epstopdf}
\usepackage{color,fancybox,graphicx}
\usepackage{url}
\usepackage{psfrag}
\usepackage{color}
\usepackage{pifont}

\usepackage{pstricks, pst-tree}

\numberwithin{equation}{section}

\newtheorem{rem}{Remark}[section]

\newtheorem{lem}{Lemma}[section]

\newtheorem{pro}{Proposition}[section]
\newtheorem{theo}{Theorem}[section]

\newtheorem{ass}{Assumption}

\newcommand{\bX}{\mathbf{X}}
\newcommand{\bx}{\mathbf{x}}
\newcommand{\bz}{\mathbf{z}}
\newcommand{\bZ}{\mathbf{Z}}

\definecolor{gris25}{gray}{0.90}

\usepackage{natbib}
\bibpunct{(}{)}{;}{a}{,}{,}

\usepackage{hyperref}
\hypersetup{
colorlinks = true,
urlcolor = blue, 
linkcolor = blue,
citecolor = blue,
}
\usepackage{enumerate}

\begin{document}

\begin{center}
{\Large 
\textbf{\textsf{Online Asynchronous Distributed Regression}}}
\medskip
\medskip
\end{center}

{\bf G\'erard Biau\footnote{Research carried out within the INRIA project ``CLASSIC'' hosted by Ecole Normale Sup{\'e}rieure and CNRS.}}\\
{\it Sorbonne Universit\'es, UPMC Univ Paris 06, F-75005, Paris, France\\
\& Institut universitaire de France}\\
\href{mailto:gerard.biau@upmc.fr}{gerard.biau@upmc.fr}
\bigskip
\newline
{\bf Ryad Zenine}
\newline
{\it Sorbonne Universit\'es, UPMC Univ Paris 06, F-75005, Paris, France}\\
\href{mailto:r.zenine@gmail.com}{r.zenine@gmail.com}
\bigskip
\begin{abstract}
\noindent {\rm 
Distributed computing offers a high degree of 
flexibility to accommodate modern learning
constraints and the ever increasing size of datasets 
involved in massive data issues. 
Drawing inspiration from the theory of distributed
computation models developed in the
context of gradient-type
optimization algorithms, we present a consensus-based
asynchronous distributed approach for nonparametric online
regression and analyze some of its asymptotic properties. Substantial numerical evidence involving up to
 28 parallel processors is provided on synthetic
 datasets to assess the excellent performance of our method, both
 in terms of computation time and prediction accuracy.
 \medskip
 
\medskip

}

\end{abstract}

\section{Introduction}
Parallel and distributed computation is currently an area of intense research
activity, motivated by a variety of factors. Examples of such factors include, but are not restricted to:
\begin{enumerate}[$(i)$]
\item Massive data 
challenges, where the sample size 
is too large to fit into a single computer or to operate 
with standard computing resources \citep[see, e.g., the discussion in][]{Jo13};
\item An increasing necessity of robustness and fault tolerance, 
that enables a system to  continue operating properly in the event of a failure; 
\item Advent of sensor, wireless and peer-to-peer networks, which obtain information 
on the state of the environment and must process it cooperatively.
\end{enumerate}
Moreover, in a growing number of distributed organizations, data are acquired sequentially and must be efficiently processed in real-time, thus avoiding batch requests or communication with a fusion center. In this sequential context, a promising way is to deal with decentralized distributed systems, in which a communication to a central processing unit is unnecessary. Yet, designing and analyzing such distributed online learning algorithms that can quickly and efficiently process large amounts of data poses several mathematical and computational challenges and, as such, is one of the exciting questions asked to the statistics and machine learning fields.
\medskip

In the present paper, we elaborate on the theory of distributed and asynchronous computation models developed in the context of deterministic and stochastic gradient-type optimization algorithms by \citet{TsBeAt86}, \citet{BeTs89}, and \citet{BlHeOlTs05}. Equipped with this theory, we present a consensus-based asynchronous distributed approach for nonparametric online regression and analyze some of its asymptotic properties. In the model that we consider, there is a number of computing entities (also called processors, agents, or workers hereafter) which perform online regression estimation by regularly updating an estimate stored in their memory. In the meanwhile, they exchange messages, thus informing each other about the results of their latest computations. Processors which receive messages use them to update directly the value in their memory by forming a convex combination. Weak assumptions are made about the frequency and relative timing of computations or message transmissions by the agents.
\medskip

The general framework is that of nonparametric regression estimation, in which an input random vector $\bX \in (\mathbb R^d, \|\cdot\|)$ is considered, and the goal is to predict the integrable random response $Y\in \mathbb R$ by assessing the regression function $r(\bx)=\mathbb E[Y|\bX=\bx]$. In the classical, offline setup \citep[see, e.g.,][]{GyKoKr02}, one is given a batch sample $\mathcal D_t=(\bX_1,Y_1), \hdots, (\bX_t,Y_t)$ of i.i.d.~random variables, distributed as  (and independent of) the prototype pair $(\bX, Y)$. The objective, then, is to use the entire dataset $\mathcal D_t$ to construct an estimate $r_t:\mathbb R^d \to \mathbb R$ of the regression function $r$. However, in many contemporary learning tasks, data arrive sequentially, and possibly asynchronously. In such cases, estimation has to be carried out online, by processing each observation one at a time and updating recursively the estimate step by step. Extending ideas from stochastic approximation \citep[][]{RoMo51,KiWo52}, \citet{Re73,Re77} introduced a kernel-type online estimate of $r$ and demonstrated some of its appealing properties.
\medskip

In a word, our architecture parallelizes several executions of R\'ev\'esz's method concurrently on a set $\{1, \hdots, M\}$ of processors that try to reach agreement on the estimation of $r(\bx)$ by asynchronously exchanging tentative values and fusing them via convex combinations. The data are sequentially allocated to the memories of these machines, so that agent $i$ sequentially receives the i.i.d.~sequence of observations $(\bX^i_1,Y^i_1), (\bX^i_2,Y^i_2), \hdots, (\bX^i_t,Y^i_{t}), (\bX^i_{t+1},Y^i_{t+1}),\hdots$, and use them to compute its estimate $r_t^i(\bx)$ of $r(\bx)$. In addition, at time $t\geq 1$, any agent $j$ in the network may transmit its current estimate $r_t^j(\bx)$ to some (possibly all or none) of the other processors. If a message from processor $j$ is received by processor $i$ $(i \neq j$) at time $t$, let $\tau^{ij}_t$ be the time that this message was sent. Therefore, the content of such a message is precisely $r^j_{\tau^{ij}_t}(\bx)$, which is simpler denoted by $r^j(\bx,\tau^{ij}_t)$. Thus, omitting details for the moment, the estimated value held by agent $i$ is updated according to the equation
\begin{equation*}
\left \{ \begin{array}{llll}
r^i_1(\bx) & = & Y^i_1&\\
r^i_{t+1}(\bx)&=&\sum_{j=1}^Ma_t^{ij} r^j(\bx,\tau^{ij}_t)+s^i_t & \mbox{for $t\geq 1$},
\end{array}
\right .
\end{equation*}
where the coefficients $a_t^{ij}$ are (deterministic) nonnegative real numbers satisfying the constraint $\sum_{j=1}^Ma^{ij}_t=1$. The term $s^i_t\equiv s^i_t(\bX^i_{t+1},Y^i_{t+1},r^i_t(\bx))$, which will be made precise later, is a local R\'ev\'esz-type computation step, to be used in evaluating the new estimate $r^i_{t+1}(\bx)$. The model can be interpreted as follows: At any time $t$, processor $i$ receives messages from other processors containing the $r^j(\bx,{\tau^{ij}_t})$'s ; it incorporates this information by forming a convex combination and adding the scalar $s^i_t$ resulting from its own local computations. The time instants $(\tau^{ij}_t)_{t \geq 1}$ are deterministic but unknown and the families $(a^{ij}_t)_{t\geq 1}$ define the weights of convex combinations.
\medskip

Under weak assumptions on the network architecture and the communication delays, we establish in Theorem \ref{bigtheo} that our distributed algorithm guarantees asymptotic consensus, in the sense that, for all $i \in \{1, \hdots, M\}$, 
$$\mathbb E\left[\int_{\mathbb R^d}\left|r_t^i(\bx)-r(\bx)\right|^2\mu(\mbox{d}\bx)\right]\to 0 \quad \mbox{as } t \to \infty,$$
where $\mu$ is the distribution of $\bX$ and $Y$ is assumed to be bounded. From a practical point of view, an important feature of the presented procedure is its ability to process, for a given time span, much more data than a single processor execution. A similar architecture has been used successfully in the context of vector quantization by \citet{Pa11}. It has also some proximity with the so-called gossip algorithms \citep[see, e.g.,][]{BoGhPrSh06,BiFoHaJa11,BiFoHaJa11-2,BiClJaMo13}. However, the primary benefit of our strategy is that it is asynchronous, which means that local processes do not have to wait at preset points for messages to become available. This allows some processors to compute faster and execute more iterations than others---a major speed advantage over synchronous executions in networks where communication delays can be substantial and unpredictable. In fact, message passing and asynchronism offer a high degree of flexibility and would make it easier to include tolerance to system failures and uncertainty. Let us finally stress that, by its online nature, the algorithm is also able to manage time-varying data loads. Online approaches avoid costly and non-scalable batch requests on the whole dataset, and offers the opportunity to incorporate new data while the algorithm is already running.
\medskip

The paper is organized as follows. We present our asynchronous distributed regression estimation strategy in Section 2, and prove its convergence in Section 3. Section 4
is devoted to numerical experiments and simulated tests which
illustrate the performance of the approach. For ease of exposition, proofs are
collected in
Section 5.
\section{A model for distributed regression}
Let $(\bX_t,Y_t)_{t \geq 1}$ be a sequence of i.i.d.~random variables, distributed as (and independent of) the generic pair $(\bX,Y)$. Assume that $\mathbb E |Y|<\infty$ and let $r(\bx)=\mathbb E[Y |\bX=\bx]$ be the regression function of $Y$ on $\bX$. In its more general form, R\'ev\'esz's recursive estimate of $r(\bx)$ \citep[][]{Re73,Re77} takes the form
\begin{equation*}
{\small
\left \{ \begin{array}{llll}
r_1(\bx) & \!\!=\!\!& Y_1&\\
r_{t+1}(\bx)&\!\!=\!\!&r_t(\bx)\left(1-\varepsilon_{t+1}K_{t+1}(\bx,\bX_{t+1})\right)+ \varepsilon_{t+1}Y_{t+1}K_{t+1}(\bx,\bX_{t+1}) & \mbox{\hspace{-0.1cm}for $t\geq 1$},
\end{array}
\right .
}
\end{equation*}
where $(K_t(\cdot,\cdot))_{t \geq 1}$ is a sequence of measurable, symmetric and nonnegative-valued functions on $\mathbb R^d\times \mathbb R^d$, and $(\varepsilon_t)_{t\geq 1}$ are positive real parameters (by convention, $K_1(\cdot,\cdot) \equiv 1$ and $\varepsilon_1=1$). A major computational advantage of this definition is that the $(t+1)$-th estimate $r_{t+1}(\bx)$ can be evaluated on the basis of the $(t+1)$-th observation $(\bX_{t+1},Y_{t+1})$ and from the $t$-th estimate $r_t(\bx)$ only, without remembering the previous elements of the sample. It should also be noted that $r_{t+1}(\bx)$ is obtained as a linear combination of the estimates $r_{t}(\bx)$ and $Y_{t+1}$, with weights $1-\varepsilon_{t+1}K_{t+1}(\bx,\bX_{t+1})$ and $\varepsilon_{t+1}K_{t+1}(\bx,\bX_{t+1})$, respectively. In a more compact form, we write
\begin{equation*}
\left \{ \begin{array}{llll}
r_1(\bx) & = & Y_1&\\
r_{t+1}(\bx)&=&r_t(\bx)- \varepsilon_{t+1}H\left(\bZ_{t+1},r_t(\bx)\right)& \mbox{for $t\geq 1$},
\end{array}
\right .
\end{equation*}
where $\bZ_t=(\bX_t,Y_t)$ and, by definition, $H(\bZ_{t+1},r_t(\bx))=r_t(\bx)K_{t+1}(\bx,\bX_{t+1})-Y_{t+1}K_{t+1}(\bx,\bX_{t+1})$. We note that time starts at $1$, and that the data acquired at time $t$ is $\bZ_{t+1}$ along with an updated estimate equal to $r_{t+1}(\bx)$.
\medskip

Various choices are possible for the function $K_t$, each leading to a different type of estimate. Letting for example 
$$K_t(\bx,\bz)=\frac{1}{h_t^d}K \left (\frac{\bx-\bz}{h_t}\right), \quad \bx,\bz \in \mathbb R^d,$$
where $K:\mathbb R^d\to \mathbb R_+$ is a symmetric kernel and $(h_t)_{t \geq 1}$ a sequence of positive smoothing parameters, results in the recursive kernel estimate, originally studied by R\'ev\'esz. However, other options are possible for $K_t$, giving rise to diverse procedures such as the recursive partitioning, series and binary tree estimates. Asymptotic properties of R\'ev\'esz-type recursive estimates have been established for diverse choices of $K_t$ by \citet{Gy81}, \citet{GyWa96,GyWa97}, \citet{Wa01}, and \citet{MoPeSl09}, just to cite a few examples \citep[see also][Chapter 25]{GyKoKr02}.
\medskip

Returning to our distributed model, consider now a set $\{1, \hdots, M\}$ of computing entities that participate in the estimation of $r(\bx)$. In this construction, the data are spread over the agents, so that processor $i$ sequentially receives the i.i.d.~sequence $(\bX^i_1,Y^i_1), (\bX^i_2,Y^i_2), \hdots, (\bX^i_t,Y^i_t), (\bX^i_{t+1},Y^i_{t+1}), \hdots$, distributed as the prototype pair $(\bX,Y)$. Processor $i$ is initialized with $r^i_1(\bx)=Y^i_1$. At time $t\geq 1$, it receives measurement $(\bX^i_{t+1},Y^i_{t+1})$ and may calculate the new estimate $r^i_{t+1}(\bx)$ by executing a local R\'ev\'esz-type step. Moreover, besides its own measurements and computations, each agent may also receive messages from other processors and combine this information with its own conclusions. The computation/combining process is assumed to be as follows:
\begin{equation}
\label{distributed}
\left \{ \begin{array}{llll}
r^i_1(\bx) & = & Y_1^i&\\
r_{t+1}^i(\bx)&=& \sum_{j=1}^Ma_t^{ij} r^j(\bx,\tau^{ij}_t)+s^i_t& \mbox{for $t\geq 1$},
\end{array}
\right .
\end{equation}
where the $a_t^{ij}$'s are nonnegative real coefficients satisfying the constraint $\sum_{j=1}^Ma^{ij}_t=1$, for all $i \in \{1, \hdots, M\}$ and all $t\geq 1$. As in the introduction, the notation $r^j(\bx,\tau^{ij}_t)$ stands for $r^j_{\tau^{ij}_t}(\bx)$. This is the value received at time $t$ by agent $i$ from agent $j$, which is thus not necessarily the most recent one. Naturally, it is assumed that the deterministic time instants $(\tau^{ij}_t)_{t\geq 1}$ satisfy $1\leq \tau^{ij}_t \leq t$: The difference $t-\tau^{ij}_t$ represents communication and possibly other types of delay, such as latency and bandwidth finiteness. As for the term $s^i_t$, it is a R\'ev\'esz-type computation step, which takes the form:
\begin{equation*}
s^i_t= \left \{ \begin{array}{ll}
-\varepsilon_{t+1}^iH\left(\bZ^i_{t+1},r^i_t(\bx)\right) &  \mbox{if $t \in T^i$}\\
0 & \mbox{otherwise}.
\end{array}
\right .
\end{equation*}
In this definition, the set $T^i$ contains all time instants where processor $i$ updates its current estimate by performing an effective estimation (accordingly, processor $i$ is called computing).  Since the combining coefficients $a^{ij}_t$ depend on $t$, the network communication topology is sometimes referred to as time-varying. 
\medskip

Noteworthy, the sequences $(\tau^{ij}_t)_{t \geq 1}$ need not to be known in advance by any agent. In fact, their knowledge is not required to execute iterations. Thus, there is no need to dispose of a shared global clock or synchronized local clocks at the processors. The time variable $t$ we refer to corresponds to an iteration counter that is needed only for analysis purposes. In particular, the computing operations may not take the same time for all processors, which is a major advantage of asynchronous algorithms.
\section{Assumptions and main results}
We establish in this section that the architecture (\ref{distributed}) guarantees $L^2$ asymptotic consensus, i.e., for all $i \in \{1, \hdots, M\}$,
$$\mathbb E\left[\int_{\mathbb R^d}\left|r_t^i(\bx)-r(\bx)\right|^2\mu(\mbox{d}\bx)\right]\to 0 \quad \mbox{as } t\to \infty,$$
where $\mu$ is the distribution of $\bX$ and $Y$ is assumed to be bounded. However, this powerful result comes at the price of assumptions on the transmission network, which essentially demand that the time between consecutive communications of processors plus communication delays are not too large. 
\medskip

We start with some basic requirements on the coefficient sequences $(a^{ij}_t)_{t \geq 1}$ and the communication delays $(\tau_t^{ij})_{t \geq 1}$, which are adapted from \citet{BlHeOlTs05}.
\begin{ass}[Convex combinations]
There exists a constant $\alpha \in (0,1]$ such that:
\begin{enumerate}[{\bf (a)}]
\item $\sum_{j=1}^M a^{ij}_t=1$,  for all $i \in \{1, \hdots, M\}$ and $t\geq 1$.
\item $a^{ii}_t \geq \alpha$, for all $i \in \{1, \hdots, M\}$ and $t\geq 1$.
\item $a^{ij}_t \in \{0\} \cup [\alpha,1]$, for all $(i,j) \in \{1, \hdots, M\}^2$ and $t\geq 1$.
\end{enumerate}
\end{ass}

Assumption {\bf 1(a)} means that the combination operated by agent $i$ at time $t$ is a weighted average of its own value and the values that it has just received from other agents. Parts {\bf (b)} and {\bf (c)} avoid degenerate situations and guarantee that messages have a lasting effect on the states of computation of their recipients. Various special cases of interest are discussed in \citet{BlHeOlTs05}. For example, in the so-called equal neighbor model, one has
\begin{equation*}
a^{ij}_t= \left \{ \begin{array}{ll}
1/{\sharp N^{i}_t}&  \mbox{if $j \in N^{i}_t$}\\
0 & \mbox{otherwise},
\end{array}
\right .
\end{equation*}
where 
$$N^{i}_t=\left\{ j \in \{1, \hdots, M\}:a^{ij}_t > 0\right\}$$
is the set of agents whose value is taken into account by processor $i$ at time $t$ (the symbol $\sharp$ stands for cardinality). Note that here the constant $\alpha$ of Assumption {\bf 1(b)} is equal to $1/M$.
\begin{ass}[Bounded communication delays]
\label{ass_bounded_com}
\begin{enumerate}[{\bf (a)}]
\item[]
\item One has $a^{ij}_t=\mathbf 1_{[i \neq j]}$, for all $(i,j) \in \{1, \hdots, M\}^2$ and $t\in T^i$.
\item If $a^{ij}_t=0$, then $\tau^{ij}_t=t$, for all $(i,j) \in \{1, \hdots, M\}^2$ and $t \geq 1$.
\item One has $\tau^{ii}_t=t$, for all $i \in \{1, \hdots, M\}$ and $t\geq 1$.
\item There exists some constant $B_1\geq0$ such that $t-B_1\leq \tau^{ij}_t \leq t$, for all $(i,j) \in \{1, \hdots, M\}^2$ and $t \geq 1$.
\end{enumerate}
\end{ass}

Assumption {\bf 2(a)} means that no combining operation is performed by processor $i$ while a R\'ev\'esz-type computation is effectively performed. This requirement is not particularly restrictive and makes sense for practical implementations. Assumption {\bf 2(b)} is just a convention: When $a^{ij}_t=0$, the value of $\tau^{ij}_t$ has no effect on the update. Assumption {\bf 2(c)} is quite natural, since an agent generally has access to its own most recent value. Finally, Assumption {\bf 2(d)} requires the communication delays $t-\tau^{ij}_t$ to be bounded by some constant $B_1$. In particular, this assumption prevents a processor from taking into account some arbitrarily old values computed by other agents.
\medskip

Denote by $\mathcal M=\{1, \hdots, M\}$ the set of processors. The communication patterns at each time step, sometimes referred to as the network communication topology, can be described in terms of a directed graph $(\mathcal M,E_t)$, with vertices $\mathcal M$ and edges $E_t$ describing links, where $(i,j) \in E_t$ if and only if $a^{ij}_t > 0$. A minimal assumption, which is necessary for consensus to be reached, entails that following an arbitrary time $t$, and for any pair of agents $(i,j)$, there is a sequence of communications through which agent $i$ will influence (directly or indirectly) the future value held by agent $j$.
\begin{ass}[Graph connectivity]
The graph $(\mathcal M,\cup_{s \geq t} E_s)$ is strongly connected (i.e., every vertex is reachable from every other vertex) for all $t\geq 1$.
\end{ass}

We also require Assumption {\bf 4} below, which complements Assumption {\bf 3} by demanding that there is a finite upper bound on the length of communicating paths.
\begin{ass}[Bounded intercommunication intervals]
There is some constant $B_2\geq0$ such that if agent $i$ communicates to $j$ an infinite number of times (that is, if $(i,j)\in E_t$ infinitely often), then, for all $t \geq 1$, $(i,j) \in E_t \cup E_{t+1} \cup \cdots \cup E_{t+B_2}$.
\end{ass}

Our last assumption is of a more technical nature. Part {\bf (a)} requires that, at any time, there is at least one processor $i$ satisfying $s^i_t\neq 0$. Thus, there are no time instants where all processors are idle. Part {\bf (b)} is mainly to simplify the presentation and could easily be refined.  Notice that this latter requirement does not necessarily imply that each processor has knowledge of $t$, i.e., access to a global clock. Assuming that the time span between consecutive updates is bounded, it is for example satisfied by taking $\varepsilon_t^i$ proportional to $n_t^i=\sharp (T^i \cap \{1, \hdots, t\})$---that is, the number of times that processor $i$ has performed a computation up to time $t$.
\begin{ass}[Idle processors and learning rate]
\begin{enumerate}[{\bf (a)}]
\item[]
\item For all $t\geq 1$, one has $\sum_{j=1}^M \mathbf 1_{[t \in T^j]} \geq 1$.
\item There exist two constants $C_1>0$ and $C_2>0$ such that, for all $i\in \{1, \hdots, M\}$ and all $t\geq 1$,
$$\frac{C_1}{t} \leq \varepsilon_t^i \leq \frac{C_2}{t}.$$
\end{enumerate}
\end{ass}

We are now in a position to state our main result. 
\begin{theo}
\label{bigtheo}
Assume that Assumptions {\bf 1-5} are satisfied. Assume that there exist a sequence $(h_t)_{t \geq 1}$ of positive real numbers and a nonnegative, nonincreasing function $L$ on $[0,\infty)$ such that $h_t \to 0$ (as $t\to \infty$), $r^dL(r)\to 0$ (as $r \to \infty$) and, for all $\bx,\bz \in \mathbb R^d$ and all $t\geq 2$,
$$h_t^d K_t(\bx,\bz)\leq L\left (\frac{\|\bx-\bz\|}{h_t} \right).$$
Assume, in addition, that $Y$ is bounded, that
$$\sup_{t,\bx,\bz} \varepsilon_t^i K_t(\bx,\bz)\leq 1 \quad \mbox{for all $i \in \{1, \hdots, M\}$},$$
and that
\begin{equation*}
\underset{t\to \infty}{\lim\inf} \int_{\mathbb R^d} K_t(\bx,\bz)\mu(\emph{d}\bz)>0 \mbox{ at $\mu$-almost all $\bx \in \mathbb R^d$}.
\end{equation*}
Then, provided $(th_t^d)_{t \geq 1}$ is nondecreasing and $\sum_{t\geq 1} \frac{1}{t^2h_t^{2d}}<\infty$, one has, for all $i \in \{1, \hdots, M\}$, 
$$\mathbb E\left[\int_{\mathbb R^d}\left|r_t^i(\bx)-r(\bx)\right|^2\mu(\emph{d}\bx)\right]\to 0 \quad \mbox{as } t \to \infty.$$
\end{theo}
\begin{rem}
To avoid ambiguity, it is assumed by convention that $K_1(\cdot, \cdot)\equiv1$ and $h_1^d\leq L(0)$.
\end{rem}
A few comments are in order. We note that apart the boundedness assumption on $Y$, this convergence is universal, in the sense that it is true for all distributions of $(\bX,Y)$. Moreover, the requirements on the function $K_t$ are mild and typically satisfied for the kernel-type choice 
$$K_t(\bx,\bz)=\frac{1}{h_t^d}\mathbf 1_{[\|\bx-\bz\|/{h_t}\leq 1]}$$
 \citep[naive kernel---see, e.g.,][]{GyWa97}, or for the choice
 $$K_t(\bx,\bz)=\frac{1}{h_t^d}\,e^{-\|\bx-\bz\|^2/{h_t^2}}$$
 \citep[Gaussian kernel---see, e.g.,][]{St70} as soon as the distribution of $\bX$ has a density. In fact, the main message of Theorem \ref{bigtheo} is that the distributed and asynchronous procedure (\ref{distributed}) retains the nice consistency properties of its centralized counterpart \citep[][]{GyKoKr02}, while allowing to handle a much larger volume of data in a reasonable time. This important feature is illustrated in the next section, where a smart implementation of the method with 28 agents allows to reduce the processing of $n=10^6$ observations from more than 10 hours to less than 30 minutes. Section 4 will also reveal that distributing the calculations has no dramatic effect in terms of estimation accuracy. Indeed, numerical evidence shows that the collaboration mechanism does not degrade the convergence rate of a single processor execution---this is a nice feature, especially in situations where the distributed architecture is imposed by physical or geographical constraint.
\section{Implementation and numerical studies}
This section is devoted to the practical analysis and performance assessment of our consensus-based asynchronous distributed regression solution. To this aim, we wrote a software in Go, an open source native concurrent programming language developed at Google Inc. We exclusively relied on the standard library shipped by the language, thus avoiding any external dependency. The code is available under an open source license at the url \texttt{http://github.com/ryadzenine/dolphin}.  We stress that our goal with this software is not to deliver a turnkey solution for distributed computing, but rather to illustrate/simulate some of the essential features and issues encountered in the analysis of distributed systems.
\medskip

We start by introducing the software general architecture and the algorithms we used. Next, we describe the experiments that were carried out and discuss the numerical results. 
\subsection{Software architecture}
The implementation of the procedure carries some challenges. To begin with, for the method to scale, one needs to manage the communication overhead. More precisely, if too many messages are exchanged during the execution of the procedure, the available network bandwidth will quickly be consumed. If this happens, the number of agents (more appropriately called workers in this section) $M$ that the procedure can effectively run in parallel will be limited, which is clearly not the desired objective. Thus, some care needs to be taken when choosing the shape of the graph $(\mathcal M,\cup_{s \geq t} E_s)$. Moreover, the asynchronous nature of the model must be preserved, which forbids the use of any synchronization mechanism in the implementation. In particular, the general software design and the underlying algorithms must ensure that concurrent writes---in a single memory space---do not happen.  
\medskip

As illustrated in Figure \ref{Arch}, our software is built on top of several workers and a messaging system. In this architecture, each worker is a software component that handles the numerical computations, while the messaging system is a distributed component that allows the workers to communicate. The messaging system is composed of several queues, where each queue keeps track of the estimate values of all workers (possibly outdated). In our setup, each queue serves two twin workers by giving them the ability to either broadcast their local estimate value or get values from the other workers. In addition, the queues are connected in a so-called ring topology to form the messaging system. In this cyclic organization, each queue is only connected to two other queues of the system: It sends data to one of them and receives information from the other, in the direction of the arrow. 
\begin{figure}[!h]
\caption{\label{Arch} Software architecture}
\begin{center}
\includegraphics[scale=0.45 ]{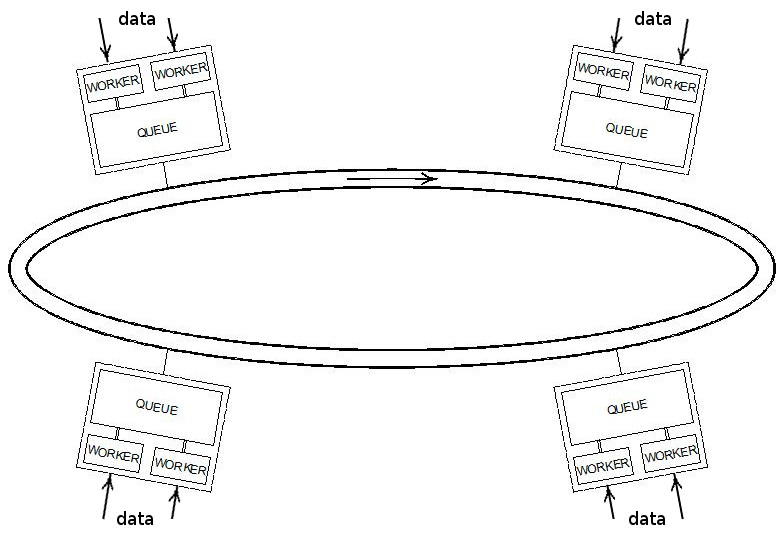}
\end{center}
\end{figure}
\medskip

The algorithm of a given worker component is easy. The worker just listens on a channel for new data to arrive, performs an iteration (\ref{distributed}), and then immediately broadcast his updated estimate value to the other workers through its serving queue. Besides, when the worker needs to perform an averaging step, it just asks its serving queue for the latest known estimate values of all the other workers, and then average only those values that were not used before. It should be noted that one worker and its twin perform iterations at the same rate, so that the averaging step of a given worker always involve at least the value transmitted by its twin to the queue. In turn, this averaging mechanism defines the families $(a^{ij}_t)_{t\geq 1}$ of weights.
\medskip

Let us now describe the way queues perform and communicate. A given queue, say $Q$, keeps in its memory the values of all the workers, possibly outdated (in our implementation, this is achieved by using a so-called associative container, or hash-map). The queue $Q$ constantly listens for a message arriving from the preceding queue in the ring, and updates its local memory once such a message is received. When this happens, $Q$ immediately sends its content to the next queue in the ring, and so on. In parallel, $Q$ listens to its twin workers and instantly updates the corresponding values each time these workers send information. At the software startup, the distributed messaging system is initialized by choosing one queue at random. This queue then sends its local memory content to the next queue in the ring, and the process starts.
\medskip

Thus, in this decentralized architecture, the workers perform at their own paces, independently of one another, and asynchronously exchange information via the messaging system. The implementation verifies  Assumptions ${\bf 1}$ to ${\bf 4}$. More precisely, the graph $(\mathcal M,\cup_{s \geq t} E_s)$ is strongly connected as required by Assumption ${\bf 2}$, and the constant $B_2$ of Assumption ${\bf 4}$ is guaranteed to be finite. Note however that $B_2$ cannot be set by hand. In fact, it depends on the performance of the underlying physical network. In addition, the initialization policy of the distributed messaging system gives us the guarantee that, at any time $t$, only two queues of the messaging system are communicating. As a consequence, only one queue is partially updating its associative container. Therefore, concurrent writes never happen in the distributed messaging system.
\medskip

The software was benchmarked on a computer with $16$ $\mbox{Intel}^{\scriptsize{\textregistered}}$ Xeon E5-4620 (2.2 GHz) processors, with four cores each. The computer is also equipped of $256$ GB of RAM. Noteworthy, each worker and each queue where launched on their own thread of execution. 
\medskip   

\subsection{ Numerical results} 
We had to fix some parameters of the estimate (\ref{distributed}) to carry out the numerical experimentations. To begin with, we consider a constant $\tau$ (which we call the metronome) whose aim is to regulate the behavior of the workers. More precisely, every worker does an averaging step for each $\tau-1$ computation steps. Precisely, for a worker $i \in \{1, \hdots, M\}$, the set $T^i$ containing the time instants where the worker performs a computation is just defined as the complementary set of $\{k \in \mathbb{N}^{\star} : k \equiv 0  \pmod \tau  \}$. 
\medskip

As for the estimate itself, we used a Gaussian kernel, of the form
 $$K_t(\bx,\bz)=\frac{1}{h_t^d}\,e^{-\|\bx-\bz\|^2/{h_t^2}},\quad \bx,\bz \in \mathbb R^d,$$
where $\|\cdot\|$ is the Euclidean norm. The smoothing parameter $h_t$ was set to $t^{- \frac{d}{d+4}}$ \citep[this choice is in line with the results of][in dimension $d=1$]{MoPeSl09}.  Finally, we let the constants $C_1$ and $C_2$ of Assumption ${\bf 3}$ be equal to $1$. We realize that other, eventually data-dependent, parameter choices are possible. However, our goal in this section is more to highlight the scaling capabilities of the algorithm (that is, its ability to deal with a large amount of data) rather than assessing its statistical performance. 
\medskip

The procedure was benchmarked on three synthetic datasets generated by the following models (we set $\bX=(X_1, \hdots, X_d)$ and let $\mathcal{N}(\mu, \sigma^2)$ be a Gaussian random variable with mean $\mu$ and variance $\sigma^2$):  
\begin{align*}
\text{\bf Model 1: } & d=2, & Y &= X^2_1 +\exp( - X^2_2).&
\\
\text{\bf Model 2: } & d=4, & Y &= X_1X_2 + X_3^2 - X_4 +  \mathcal{N}(0,0.05). &
\\
\text{\bf Model 3: } & d=4, & Y &= \mathbf 1_{[X_1 > 0]} + \mathbf 1_{[X_4 -X_2 > 1 + X_3]} + X^3_2 +\exp( - X^2_2) \\
&\quad & & \quad + \mathcal{N}(0,0.05). &
\end{align*}
In order to keep $Y$ bounded, all values of $|Y|$ above 1 were discarded.  For each model, two designs were considered: Uniform over $(0,1)^d$ and Gaussian with mean $0$ and covariance matrix $\varSigma$ with $\varSigma_{ij} = 2^{-\mid i-j\mid}$. In addition, the benchmarks were carried out with a number of workers $M$ ranging from $1$ to $28$. Since a typical workstation has between 4 and 8 processors, it should be noted that these experimental values of $M$ cover a wide range of practical cases. In order to assess the impact of the averaging step, two values were tested for the metronome: $\tau$= $2$ (high frequency message passing) and $\tau=M^2$ (low frequency message passing).
\medskip  

Each simulated dataset contains $10^6$ observations and is split into $M$ training sets $\mathcal A_i$, $i \in \{1, \hdots, M\}$, and a test set $\mathcal T$. The test set contains $20\%$ of the data and the remaining examples are uniformly distributed between the training sets $\mathcal A_i$. For every pair $(\bx, y)$ in $\mathcal T$, and for every worker $i$, we train the estimate $r^i(\bx)$ on the data sequentially acquired from $\mathcal A_i$, and finally evaluate the $L^2$ error of $i$ via the formula
\[
	\mbox{\sc err}_i=\sum_{(\bx,y) \in \mathcal T} \left(y-r^i(\bx)\right)^2.
\]     
 For each experiment and each $i$, $\mbox{\sc err}_i$ was measured every $5$ seconds by stopping the corresponding worker,  which was then immediately resumed. 
\medskip  

Figure \ref{tabtimes} shows the computing time with different values of $M$ and $\tau$ for {\bf Model 1} and the uniform design. This figure also contains a bar labeled ``{\it Optimal}'' which corresponds to the time a parallelized procedure with no communication overhead would take to process the entire dataset. It is defined as the time the estimate with one single worker takes to process the entire dataset divided by the number of workers. 
\begin{figure}[!h]
\caption{Relative computing times ({\bf Model 1}, uniform design)}
\label{tabtimes}
\includegraphics[scale=0.64]{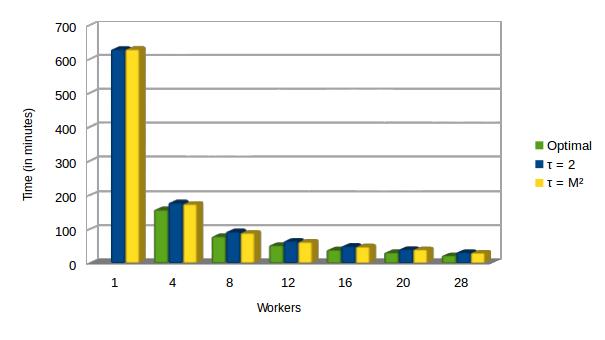}
\end{figure}
We notice that the computing time significantly decreases when more workers are available. In fact, the procedure can shrink the computing time by an order of magnitude, passing from 10:30 hours of calculation with 1 processor to less than 30 minutes with 28 workers! We also see that the decrease of computing time is close to the best one could get by adding more computing power when $M$ is small. However, this is no more true when $M$ gets larger, because of the overhead introduced by the averaging step, which grows linearly with $M$. Finally, we also observe the low impact of the different values of $\tau$ on the final computing time. This set of remarks also holds for {\bf Model 2}, {\bf Model 3}, and the Gaussian design (not shown). This is coherent, since the model choice has in fact no impact on the scaling properties of the algorithm.
\medskip

Figure \ref{figmodel2_u} depicts a typical temporal evolution of the empirical $L^2$ error (av\-er\-ag\-ed over all workers) for {\bf Model 1} and the uniform design (the same patterns are observed for the other models). 
\begin{figure}[!!h]
\caption{A typical evolution of the empirical error ({\bf Model 1}, unifo\-rm design)}
\label{figmodel2_u}
\includegraphics[scale=0.74]{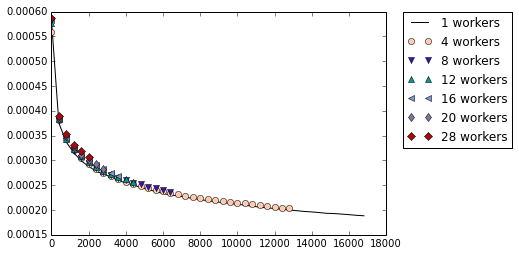}
\end{figure}
As predicted by the theory, consensus and convergence happen for every value of $M$. Figures \ref{figdiff2_u}-\ref{figdiffM_u} show with more details the relative effects of our model when compared to a basic R\'ev\'esz-type estimate (obtained by taking $M=1$). Denoting by $\mbox{\sc err}$ the empirical error of this basic online estimate and by $\mbox{\sc err}_1,\hdots,\mbox{\sc err}_M$ the respective errors of the $M$ workers of the asynchronous distributed solution, we computed every $5$ seconds the quantity
$$
\mbox{\sc relative gain}=\frac{\mbox{\sc err} - \frac{1}{M}\sum_{i=1}^{M}\mbox{\sc err}_i}{\mbox{\sc err}}.$$
(A small negative value of the {\sc relative gain} means that the distributed model performs almost as well as the non-distributed one in terms of estimation accuracy). The main message here is that our distributed procedure does not seem to deteriorate the R\'ev\'esz-type estimate ($M=1$) convergence rate---the degradation is typically negligible, of the order of $2\%$ with some peaks around $5 \%$ in {\bf Model 1}. 
\begin{figure}[!!h]
\caption{{\sc relative gain} ({\bf Model 1}, uniform design)}
\label{figdiff2_u}
\includegraphics[scale=0.44]{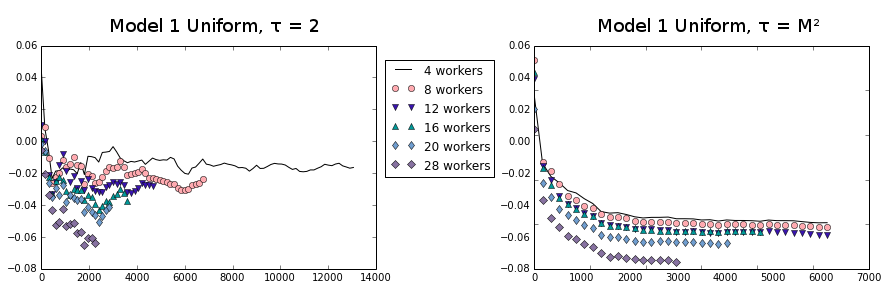}
\end{figure}
\begin{figure}[!!h]
\caption{{\sc relative gain} ({\bf Model 2}, Gaussian design)}
\label{figdiff2_nn}
\includegraphics[scale=0.44]{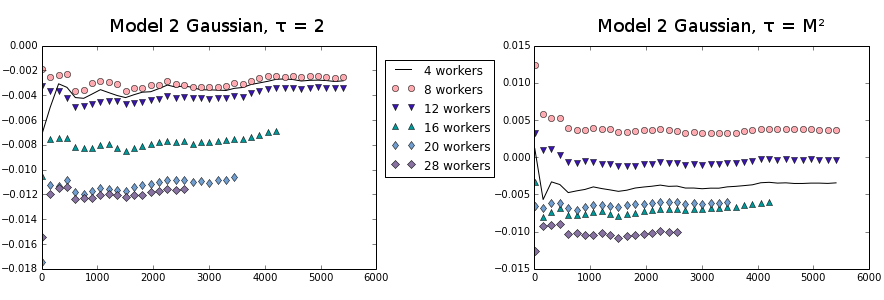}
\end{figure}
\begin{figure}[!!h]
\caption{{\sc relative gain} ({\bf Model 3}, Gaussian design)}
\label{figdiffM_u}
\includegraphics[scale=0.44]{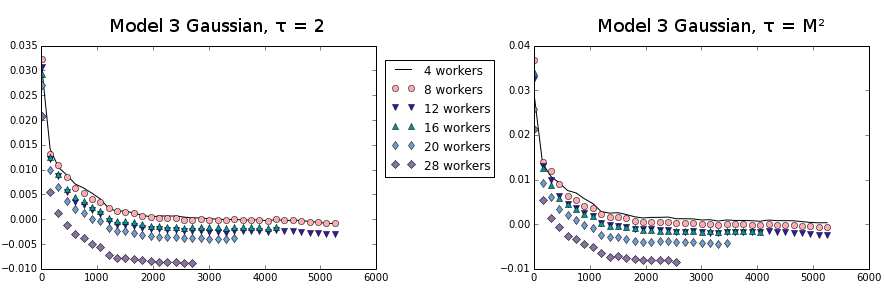}
\end{figure}
\section{Proof of Theorem \ref{bigtheo}}
\subsection{Some preliminary results}
Procedure (\ref{distributed}) falls in the general model for distributed and asynchronous computation presented by \citet{TsBeAt86}, and analyzed by these authors in the context of deterministic and stochastic gradient-type algorithms \citep[see also][]{Ts84,BeTs89}. This model has the following format:
\begin{equation}
\label{tsitsi}
z^i_{t+1}=\sum_{j=1}^M a^{ij}_tz^j(\tau^{ij}_t)+s^i_t, \quad \mbox{for all } i \in \{1, \hdots, M\} \mbox{ and all } t \geq 1,
\end{equation}
where the value $z^i_t$ is held by agent $i$ at time $t$, starting with some initial $z^i_1$ (as before, we let $z^j(\tau^{ij}_t)=z^j_{\tau^{ij}_t}$).
The term $s^i_t$ is a general computation step, whose form is not specified in this subsection. 
\medskip

Equation (\ref{tsitsi}), which defines the structure of the algorithm, is a linear system driven by the steps $(s^i_t)_{t\geq 1}$. In the special case where communication delays are zero, we have $\tau^{ij}_t=t$, and (\ref{tsitsi}) becomes a linear system with state vector $(z^1_t, \hdots, z^M_t)$. In general, however, the presence of communication delays necessitates a more involved analysis. Exploiting linearity, it is easy to conclude that, for each $t\geq 1$,  there exist scalars $\phi^{ij}(t,0), \hdots, \phi^{ij}(t,t-1)$ such that
\begin{equation}
\label{301213}
z^i_t=\sum_{j=1}^M \phi^{ij}(t,0)z^j_1+\sum_{\tau=1}^{t-1}\sum_{j=1}^M \phi^{ij}(t,\tau)s^j_\tau.
\end{equation}
The coefficients $(\phi^{ij}(t,\tau))_{t\geq 1, 0\leq \tau \leq t-1}$ do not depend upon the values taken by the computation terms $s^i_t$. They are determined by the sequence of transmission and reception times and the combining coefficients. Consequently, they are unknown, in general. Nevertheless, they have the following qualitative properties:
\begin{lem}[\citealp{TsBeAt86}] 
\label{lemmetsitsi}
Let the arrays $(\phi^{ij}(t,\tau))_{t\geq 1, 0\leq \tau \leq t-1}$ be defined as in (\ref{301213}). Then:
\begin{enumerate}
\item If Assumption {\bf 1} is satisfied, then $\phi^{ij}(t,\tau)\geq 0$  and $0 \leq\sum_{j=1}^M\phi^{ij}(t,\tau)\leq 1$, for all $(i,j) \in \{1, \hdots, M\}^2$ and all $0\leq \tau \leq t-1$.
\item Assume that Assumptions {\bf 1-4} are satisfied. Then the following statements are true:
\begin{enumerate}
\item For all $(i,j) \in \{1, \hdots, M\}^2$ and all $\tau\geq 0$, the limit of $\phi^{ij}(t,\tau)$ as $t$ tends to infinity exists. This limit is independent of $i$ and is denoted by $\phi^j_\tau$. 
\item There exists a constant $\eta>0$ such that $\phi^j_\tau\geq \eta$, for all $j \in \{1, \hdots, M\}$ and all $\tau \geq 0$.
\item There exists a constant $A>0$ and $\rho \in(0,1)$ such that, for all $(i,j)\in \{1, \hdots,M\}^2$ and all $0\leq \tau \leq t-1$,
$$\left | \phi^{ij}(t,\tau)-\phi^j_\tau\right|\leq A \rho^{t-\tau}.$$
\end{enumerate}
\end{enumerate}
\end{lem}
\begin{rem}
\label{normalization}
It is a simple but useful exercise to prove that $\sum_{j=1}^M\phi^{ij}(t,0)=1$, for all $i \in \{1, \hdots, M\}$ and all $t\geq1$. Consequently, letting $t\to \infty$, we see that $\sum_{j=1}^M\phi_0^j=1$. Also, for all $\tau \geq 0$, $0 \leq\sum_{j=1}^M\phi^{j}_{\tau}\leq 1$.
\end{rem}

The pioneering ideas of \citet{TsBeAt86} have been further explored by \citet{BlHeOlTs05} in the simplified context of a so-called agreement algorithm of the form
\begin{equation}
\label{agreement}
z^i_{t+1}=\sum_{j=1}^M a^{ij}_tz^j(\tau^{ij}_t), \quad \mbox{for all } i \in \{1, \hdots, M\} \mbox{ and all } t \geq 1.
\end{equation}
The following result expresses the fact that Assumptions {\bf 1-4} are sufficient for the agents of model (\ref{agreement}) to reach an asymptotic consensus.
\begin{theo}[\citealp{BlHeOlTs05}]
\label{theoblondel}
Consider the agreement model (\ref{agreement}), and assume that Assumptions {\bf 1-4} are satisfied. Then there exists a consensus value $z^{\star}$ (independent of $i$) such that, for all $i\in \{1, \hdots, M\}$,
$$z^i_t \to z^{\star}\quad \mbox{ as }t \to \infty.$$
\end{theo}
Let us consider again the general model (\ref{tsitsi}). Fix a time instant $t_0\geq 1$, and assume that the processors stop computing after time $t_0$ (that is, $s^i_t=0$ for all $t\geq t_0$), but keep communicating and combining. Then equation (\ref{tsitsi}) takes the form
\begin{equation*}
z^i_{t+1}=\sum_{j=1}^M a^{ij}_tz^j(\tau^{ij}_t), \quad \mbox{for all } i \in \{1, \hdots, M\} \mbox{ and all } t \geq t_0.
\end{equation*}
Thus, in that case, Theorem \ref{theoblondel} shows that the iterative process asymptotically reaches a consensus value, depending upon $t_0$. Call this limiting scalar $z^{\star}_{t_0}$. Thus, according to Lemma \ref{lemmetsitsi}, taking the limit in $t$ on both sides of identity (\ref{301213}), we have
\begin{equation}
\label{301213-2}
z^{\star}_t=\sum_{j=1}^M\phi^j_0z^j_1+\sum_{\tau=1}^{t-1}\sum_{j=1}^M\phi^j_\tau s^j_\tau, \quad \mbox{for all }t \geq 1.
\end{equation}
We note that the definition of $z^{\star}_t$ does not imply any assumption on the computation terms. Also, one easily verifies that the agreement sequence $(z^{\star}_t)_{t \geq 1}$ satisfies the following recursion formula:
\begin{equation}
\label{av1}
\left \{ \begin{array}{llll}
z^{\star}_{1}& = & \sum_{j=1}^M\phi^j_0z_1^{j}&\\
z^{\star}_{t+1}&=&z^{\star}_t+\sum_{j=1}^M\phi^j_t s^{j}_{t}&\mbox{for $t \geq 1$}.
\end{array}
\right .
\end{equation}
The scalar $z^{\star}_t$ is the value at which all processors would asymptotically agree if they were to stop computing (but keep communicating and combining) at a time $t$. It may be viewed as a concise global summary of the state of computation at time $t$, in contrast to the $z^i_t$'s, which are the local states of computation.  
\subsection{Proof of Theorem \ref{bigtheo}}
The proof of Theorem \ref{bigtheo} starts with the observation that the distributed architecture (\ref{distributed}) under study is but a special case of model (\ref{tsitsi}), with
\begin{equation*}
s^i_t= \left \{ \begin{array}{ll}
-\varepsilon_{t+1}^iH\left(\bZ^i_{t+1},r_t^i(\bx)\right) &  \mbox{if $t \in T^i$}\\
0 & \mbox{otherwise},
\end{array}
\right .
\end{equation*}
where we recall that $\bZ^i_t=(\bX^i_t,Y^i_t)$ and $H(\bZ^i_{t+1},r^i_t(\bx))=r_t^i(\bx)K_{t+1}(\bx,\bX^i_{t+1})-Y^i_{t+1}K_{t+1}(\bx,\bX^i_{t+1})$. The set $T^i$ contains all time instants where processor $i$ is effectively computing. In particular, identity (\ref{av1}) guarantees the existence of an agreement sequence $(r_t^{\star}(\bx))_{t \geq 1}$ satisfying the recursion
\begin{equation}
\label{180114}
\left \{ \begin{array}{llll}
r^{\star}_1(\bx) &=&\sum_{j=1}^M\phi^j_0Y^j_1  &\\
r^{\star}_{t+1}(\bx)&=&r^{\star}_t(\bx)-\sum_{j=1}^M \mathbf 1_{[t \in T^j]}\phi^j_t\varepsilon_{t+1}^jH\left(\bZ^j_{t+1},r^j_t(\bx)\right)& \mbox{for } t \geq 1,
\end{array}
\right .
\end{equation}
where the $[0,1]$-valued functions $\phi^j_t$ are defined in Lemma \ref{lemmetsitsi}. The limiting sequence $(r^{\star}_t(\bx))_{t \geq 1}$ plays a central role in the proof of Theorem \ref{bigtheo}. 
\medskip

The following lemma ensures that whenever $Y$ is bounded, then so are the sequences $(r^i_t(\bx))_{t \geq 1}$ and $(r^{\star}_t(\bx))_{t \geq 1}$.
\begin{lem}
\label{tKtlemma}
Assume that Assumptions {\bf 1-4} are satisfied. Assume, in addition, that $|Y|\leq \gamma$, and that
\begin{equation}
\label{tKt}
\sup_{t,\bx,\bz} \varepsilon^i_t K_t(\bx,\bz)\leq1,\quad \mbox{for all $i \in \{1, \hdots, M\}$}.
\end{equation}
Then $\sup_{\bx \in \mathbb R^d}|r^i_t(\bx)| \leq \gamma$, for all $i \in \{1, \hdots, M\}$ and all $t\geq 1$. Moreover, $\sup_{\bx \in \mathbb R^d}|r^{\star}_t(\bx)| \leq \gamma$, for all $t\geq 1$.
\end{lem}
\begin{proof}[Proof of Lemma \ref{tKtlemma}]
We know, by definition of the set $T^i$, that $s^i_t=0$ whenever $t \not\in T^i$. Furthermore, according to Assumption {\bf 2(a)}, $a^{ij}_t=\mathbf 1_{[i \neq j]}$ for $t\in T^i$. Thus,
\begin{equation*}
{\small
\left \{ \begin{array}{llll}
r^i_{t+1}(\bx) &\!\!=\!\!& r^i_t(\bx) \left[1-\varepsilon_{t+1}^iK_{t+1}(\bx,\bX^i_{t+1})\right]+\varepsilon_{t+1}^iY^i_{t+1}K_{t+1}(\bx,\bX^i_{t+1})&\mbox{if }t\in T^i\\
r^i_{t+1}(\bx)&\!\!=\!\!& \sum_{j=1}^M a^{ij}_t r^j(\bx,\tau^{ij}_t)& \mbox{\hspace{-0.3cm}otherwise}.
\end{array}
\right .
}
\end{equation*}
The first statement follows easily from the boundedness of $Y$ and inequality (\ref{tKt}). The  second claim is then
an immediate consequence of the definition of $r^{\star}_{t_0}(\bx)$ as the limit of any of the $r^i_t(\bx)$'s if the processors stop computing after time $t_0$.
\end{proof}

The main idea of the proof is to establish an equivalent of Theorem \ref{bigtheo} with $r_t^{\star}$ in place of $r_t^i$. To this aim, we start by rewriting iteration (\ref{180114}) in the following form:
\begin{equation*}
\left \{ \begin{array}{llll}
r^{\star}_1(\bx) &\!\!\!\!=\!\!\!\!&\sum_{j=1}^M\phi^j_0Y^j_1  &\\
r^{\star}_{t+1}(\bx) &\!\!\!\!=\!\!\!\!&r^{\star}_t(\bx)-\sum_{j=1}^M \mathbf 1_{[t \in T^j]}\phi^j_t\varepsilon_{t+1}^jH\left(\bZ^j_{t+1},r^{\star}_t(\bx)\right)+\Delta_{t+1}(\bx)& \mbox{\hspace{-0.15cm}for $t \geq 1$},
\end{array}
\right .
\end{equation*}
where
$$\Delta_{t+1}(\bx)\stackrel{\mbox{\tiny def}}{=}\sum_{j=1}^M \mathbf 1_{[t \in T^j]}\phi^j_t\varepsilon_{t+1}^j\left [H\left(\bZ^j_{t+1},r^{\star}_t(\bx)\right)-H\left(\bZ^j_{t+1},r^{j}_t(\bx)\right)\right].$$
The crucial step is to observe that, for all $t\geq 1$,
\begin{equation}
\label{M+M}
r_t^{\star}(\bx)=m_t^{\star}(\bx) + M_t(\bx),
\end{equation}
where $m_t^{\star}(\bx)$ obeys the recursion
\begin{equation}
\label{QJSM}
\left \{ \begin{array}{llll}
m^{\star}_1(\bx) &=&\sum_{j=1}^M\phi^j_0Y^j_1  &\\
m^{\star}_{t+1}(\bx) &=&m^{\star}_t(\bx)-\sum_{j=1}^M \mathbf 1_{[t \in T^j]}\phi^j_t\varepsilon_{t+1}^jH\left(\bZ^j_{t+1},m^{\star}_t(\bx)\right)& \mbox{for } t \geq 1,
\end{array}
\right .
\end{equation}
and 
\begin{equation}
\label{MT}
M_t(\bx)=\sum_{\tau=2}^{t}\bigg [\Delta_{\tau}(\bx)\prod_{\ell=\tau+1}^t \Big(1-\sum_{j=1}^M\mathbf 1_{[\ell-1 \in T^j]}\phi^j_{\ell-1}\varepsilon^j_{\ell}K_{\ell}(\bx,\bX^j_{\ell})\Big)\bigg]
\end{equation}
(by convention, an empty sum is $0$ and a void product is $1$). In view of decomposition (\ref{M+M}), the rest of the proof is naturally divided into two steps: Firstly we establish $L^2$ consistency of the intermediary estimate $m^{\star}_t(\bx)$ towards $r(\bx)$ (Proposition \ref{HZ}), and secondly we show that the reminder term $M_t(\bx)$ tends to zero in $L^2$ (Proposition \ref{l12}).
\medskip

An easy induction reveals that the intermediary estimate $m_t^{\star}(\bx)$ is
$$m_t^{\star}(\bx)=\sum_{i=1}^M\sum_{\tau=1}^t W_{t,\tau}^i(\bx)Y_{\tau}^i,$$
where, for any processor $i \in \{1, \hdots, M\}$, any time instant $t\geq 1$, and all $1 \leq \tau \leq t$,
$$
W^i_{t,\tau}(\bx)=\mathbf 1_{[\tau-1 \in T^i]}\phi^i_{\tau-1}\varepsilon^i_{\tau}K_{\tau}(\bx,\bX^i_{\tau})\prod_{\ell=\tau+1}^t\Big(1-\sum_{j=1}^M\mathbf 1_{[\ell-1 \in T^j]}\phi^j_{\ell-1}\varepsilon^j_{\ell}K_{\ell}(\bx,\bX^j_{\ell})\Big)
$$
(by convention, $\mathbf 1_{[0 \in T^i]}=1$, $\varepsilon_1^i=1$, and $K_1(\cdot,\cdot) \equiv 1$). The weights $W^i_{t,\tau}(\bx)$ are nonnegative random variables which do not depend upon the values of the $Y_t$'s. Moreover, it is easy to check that they satisfy the normalizing condition
$$\sum_{i=1}^M\sum_{\tau=1}^tW^i_{t,\tau}(\bx)=1.$$
(Recall that $\sum_{j=1}^M\phi_0^j=1$ and $0 \leq \sum_{j=1}^M\phi_{\tau}^jÊ\leq 1$ for $\tau \geq 0$ ; see Remark \ref{normalization}.) Thus, the good news is that the estimate $m^{\star}(\bx)$ is but a special form of a locally weighted
average estimate \citep[see, e.g.,][Chapter 2]{GyKoKr02}. According to Stone's theorem
(\citealp[][]{St77}, and \citealp[Chapter 4 in][]{GyKoKr02}), $L^2$ consistency of $m^{\star}(\bx)$ holds if the following three conditions are satisfied:
\begin{enumerate}[$(i)$]
\item There is a constant $c$ such that, for every nonnegative Borel measurable function $f:\mathbb R^d \to \mathbb R$ satisfying $\mathbb E f(\bX)<\infty$,
$$\mathbb E \left [ \sum_{i=1}^M\sum_{\tau=1}^tW^i_{t,\tau}(\bX) f(\bX^i_{\tau})\right] \leq c\,\mathbb E f(\bX),\quad \mbox{for all }t \geq 1.$$
\item  For all $a>0$,
$$\mathbb E \bigg [\sum_{i=1}^M \sum_{\tau=1}^t W^i_{t,\tau}(\bX)\mathbf 1_{[\|\bX^i_{\tau}-\bX\|>a]}\bigg ]  \to 0\quad \mbox{as } t \to \infty.$$
\item One has
$$\mathbb E \bigg[\sum_{i=1}^M\sum_{\tau=1}^t\left (W^i_{t,\tau}(\bX)\right)^2\bigg] \to 0\quad \mbox{as } t \to \infty.$$
\end{enumerate}
\begin{pro}
\label{HZ}
Assume that there exist a sequence $(h_t)_{t \geq 1}$ of positive real numbers and a nonnegative, nonincreasing function $L$ on $[0,\infty)$ such that $h_t \to 0$ (as $t\to \infty$), $r^dL(r)\to 0$ (as $r \to \infty$) and, for all $\bx,\bz \in \mathbb R^d$ and all $t\geq 2$,
\begin{equation}
\label{plus+}
h_t^d K_t(\bx,\bz)\leq L\left (\frac{\|\bx-\bz\|}{h_t} \right).
\end{equation}
Assume, in addition, that
$$\sup_{t,\bx,\bz} \varepsilon_t^i K_t(\bx,\bz)\leq1 \quad \mbox{for all $i \in \{1, \hdots, M\}$},$$
and that
\begin{equation}
\label{25.4}
\underset{t\to \infty}{\lim\inf} \int_{\mathbb R^d} K_t(\bx,\bz)\mu(\emph{d}\bz)>0 \mbox{ at $\mu$-almost all $\bx \in \mathbb R^d$}.
\end{equation}
Then, provided $th_t^d \to \infty$, one has, for all $i \in \{1, \hdots, M\}$, 
$$\mathbb E\left[\int_{\mathbb R^d}\left|m_t^{\star}(\bx)-r(\bx)\right|^2\mu(\emph{d}\bx)\right]\to 0 \quad \mbox{as } t \to \infty.$$
\end{pro}
\begin{proof}[Proof of Proposition \ref{HZ}]
The arguments are adapted from the proof of Theorem 25.1 in \citet{GyKoKr02}, which offers a similar result for the centralized version. We proceed by checking Stone's conditions $(i)$-$(iii)$ of consistency. 
\medskip

To show $(i)$, fix $f$ a nonnegative integrable function on $\mathbb R^d$, and define $\bar m_t^{\star}(\bx)$ by iteration (\ref{QJSM}), with $\bar m^{\star}_1(\bx)=\sum_{j=1}^M \phi_0^jf(\bX_1^j)$ in place of $\sum_{j=1}^M \phi_0^jY_1^j$, and $(\bX^j_{t+1},f(\bX_{t+1}^j))$ in place of $(\bX^j_{t+1},Y_{t+1}^j)$. We shall prove that
\begin{equation}
\label{S1}
\mathbb E \bar m_t^{\star}(\bX)= \mathbb E f(\bX),
\end{equation}
which implies $(i)$ with $c=1$. To establish (\ref{S1}), denote by $\mathcal F_t$ the $\sigma$-algebra generated by $(\bX_1,Y_1), \hdots, (\bX_t,Y_t)$. Then
\begin{align*}
\mathbb E\left [ \bar m_{t+1}^{\star}(\bX) | \mathcal F_t\right] &= \mathbb E\left [ \bar m_{t}^{\star}(\bX) | \mathcal F_t\right] \\
& + \sum_{j=1}^M \mathbf 1_{[t \in T^j]}\phi^j_t\varepsilon_{t+1}^j \mathbb E \big [ \left( f(\bX^j_{t+1})- \bar m_{t}^{\star}(\bX)\right)K_{t+1}(\bX,\bX^j_{t+1})|\mathcal F_t\big]\\
&= \mathbb E\left [ \bar m_{t}^{\star}(\bX) | \mathcal F_t\right]  \\
& + \varepsilon^{\star}_{t+1}\int_{\mathbb R^d}\int_{\mathbb R^d} \left( f(\bx)-\bar m_{t}^{\star}(\bz)\right) K_{t+1}(\bz,\bx)\mu(\mbox{d}\bx)\mu(\mbox{d}\bz),
\end{align*}
where, to lighten notation a bit, we set
$$\varepsilon^{\star}_{t+1}= \sum_{j=1}^M \mathbf 1_{[t \in T^j]}\phi^j_t\varepsilon_{t+1}^j.$$
Thus,
\begin{align*}
&\mathbb E\left [ \bar m_{t+1}^{\star}(\bX) | \mathcal F_t\right] \\
&\quad = \mathbb E\left [ \bar m_{t}^{\star}(\bX) | \mathcal F_t\right]  + \varepsilon^{\star}_{t+1} \int_{\mathbb R^d}\int_{\mathbb R^d} \left( f(\bx)-\bar m_{t}^{\star}(\bx)\right) K_{t+1}(\bx,\bz)\mu(\mbox{d}\bx)\mu(\mbox{d}\bz)\\
& \qquad (\mbox{by symmetry of $K_t(\cdot,\cdot)$})\\
&\quad = \mathbb E\left [ \bar m_{t}^{\star}(\bX) | \mathcal F_t\right] + \varepsilon^{\star}_{t+1} \mathbb E \big [ \left(f(\bX)- \bar m_{t}^{\star}(\bX)\right)K_{t+1}(\bX,\bX^1_{t+1})\big].
\end{align*}
Therefore, taking expectation on both sides of the equality, and noting that $\mathbb E\bar m^{\star}_1(\bX)=\mathbb Ef(\bX)$, we see that
\begin{equation*}
{\small
\left \{ \begin{array}{lll}
\mathbb E \bar m_{1}^{\star}(\bX)&\!\!\!=\!\!\!& \mathbb E f(\bX)\\
\mathbb E \bar m_{t+1}^{\star}(\bX)&\!\!\!=\!\!\!& \mathbb E \bar m_{t}^{\star}(\bX) + \varepsilon^{\star}_{t+1} \mathbb E \big [ \left(f(\bX_{t+1})- \bar m_{t}^{\star}(\bX)\right)K_{t+1}(\bX,\bX^1_{t+1})\big] \quad \mbox{\hspace{-0.3cm}for $t\geq 1$}.
\end{array}
\right .
}
\end{equation*}
Next, let the sequence $(f^{\star}_t(\bx))_{t\geq 1}$ be defined by the iteration
\begin{equation*}
\left \{ \begin{array}{lll}
f^{\star}_1(\bx)&=& f(\bx)\\
f^{\star}_{t+1}(\bx)&=& f^{\star}_{t}(\bx)+\varepsilon^{\star}_{t+1} \left( f(\bx)-f^{\star}_t(\bx)\right)K_{t+1}(\bx,\bX^1_{t+1}) \quad \mbox{for $t\geq 1$}.
\end{array}
\right .
\end{equation*}
Clearly, $f_t^{\star}(\bx)=f(\bx)$ for all $t\geq 1$, and
$$\mathbb E f^{\star}_{t+1}(\bX)=\mathbb E f^{\star}_{t}(\bX)+\varepsilon^{\star}_{t+1}\mathbb E\big[ \left( f(\bX)-f^{\star}_t(\bX)\right)K_{t+1}(\bX,\bX^1_{t+1})\big].$$
Consequently, the sequences $(\mathbb E \bar m_{t}^{\star}(\bX))_{t \geq 1}$ and $(\mathbb E f^{\star}_{t}(\bX))_{t \geq 1}$ satisfy the same iteration, and thus
$$\mathbb E \bar m_{t}^{\star}(\bX)=\mathbb E f^{\star}_{t}(\bX)=\mathbb E f(\bX).$$
This proves (\ref{S1}).
\medskip

Secondly, for $(iii)$, we set
$$p_t(\bx)=\int_{\mathbb R^d}K_t(\bx,\bz)\mu(\mbox{d}\bz).$$
For each $i \in \{1, \hdots, M\}$, each $1\leq \tau \leq t$, and all $\bx \in \mathbb R^d$, we have, using the properties of $K_t(\cdot,\cdot)$ and Assumption {\bf 5(b)},  
$$W^i_{t,\tau}(\bx) \leq \frac{C_2L(0)}{\tau h_{\tau}^d}.$$
Also, recalling that $\eta$ is the positive constant of Lemma \ref{lemmetsitsi}, 
\begin{align*}
\mathbb E W^i_{t,\tau}(\bx) &\leq  \frac{C_2L(0)}{\tau h_{\tau}^d}\exp \left[-{C_1\eta }\sum_{\ell=\tau+1}^t\frac{1}{\ell}\Big (\mathbb E K_{\ell}(\bx,\bX)\sum_{j=1}^M\mathbf 1_{[\ell-1 \in T^j]}\Big)\right]\\
& \leq   \frac{C_2L(0)}{\tau h_{\tau}^d}\exp \left[-{C_1\eta }\sum_{\ell=\tau+1}^t\frac{p_{\ell}(\bx)}{\ell}\right].
\end{align*}
In the second inequality, we used Assumption {\bf 5(a)}. Hence, evoking condition (\ref{25.4}), we deduce that, for each $i \in \{1, \hdots, M\}$ and $\tau \geq 1$, and for $\mu$-almost all $\bx \in \mathbb R^d$,
\begin{equation}
\label{25.8}
\mathbb E W^i_{t,\tau}(\bx) \to 0\quad \mbox{ as } t\to \infty.
\end{equation}
Since $\mathbb EW^i_{t,\tau}(\bx)\leq 1$, this implies, by the Lebesgue dominated convergence theorem, that
$\mathbb E W^i_{t,\tau}(\bX) \to 0$ as well.
Thus, putting all the pieces together, we conclude that
$$\mathbb E \bigg[\sum_{i=1}^M\sum_{\tau=1}^t \left(W_{t,\tau}^i(\bX)\right)^2 \bigg]\leq C_2L(0)\sum_{i=1}^M\sum_{\tau=1}^t \frac{\mathbb E W_{t,\tau}^i(\bX)}{\tau h_{\tau}^d} \to 0$$
by the condition $th_t^d\to \infty$ and the Toeplitz lemma \citep[see, e.g., Problem A.5 in][]{GyKoKr02}.
\medskip

To prove Stone's condition $(ii)$, it is enough to establish that for all $i \in \{1, \hdots, M\}$, all $a>0$, and $\mu$-almost all $\bx \in \mathbb R^d$, 
$$\mathbb E \bigg [\sum_{\tau=1}^t W^i_{t,\tau}(\bx)\mathbf 1_{[\|\bX^i_{\tau}-\bx\|>a]}\bigg ]  \to 0\quad \mbox{as } t \to \infty.$$
To this aim, first observe that by condition (\ref{25.4}), for $\mu$-almost all $\bx$,
there exists $p(\bx)>0$ and a large enough time $t_0(\bx)$ such that, for all $t\geq t_0(\bx)$, $p_t(\bx) \geq p(\bx)$. Thus, using (\ref{25.8}), we see that it is enough to show that, for these $\bx$, 
$$\mathbb E \left [ \sum_{\tau=t_0(\bx)}^t W^i_{t,\tau}(\bx)\mathbf 1_{[\|\bX^i_{\tau}-\bx\|>a]}\right] \to 0.$$
But
\begin{align*}
&\mathbb E \left [\sum_{\tau=t_0(\bx)}^t W^i_{t,\tau}(\bx)\mathbf 1_{[\|\bX^i_{\tau}-\bx\|>a]}\right] \\
& \quad =  \sum_{\tau=t_0(\bx)}^t \mathbb E W^i_{t,\tau}(\bx)\times\frac{\mathbb E \left [K_{\tau}(\bx,\bX^i_{\tau})\mathbf 1_{[\|\bX^i_{\tau}-\bx\|>a]}\right]}{\mathbb E  K_{\tau}(\bx,\bX^i_{\tau})}\\
&  \quad \leq \sum_{\tau=t_0(\bx)}^t \mathbb E W^i_{t,\tau}(\bx)\times\frac{h_{\tau}^{-d}L(a/h_{\tau})}{p_{\tau}(\bx)}\\
& \quad \leq \sum_{\tau=t_0(\bx)}^t \mathbb E W^i_{t,\tau}(\bx) \times\frac{h_{\tau}^{-d}L(a/h_{\tau})}{p(\bx)}\\
& \qquad \to 0,
\end{align*}
by the fact that $h_{t}^{-d}L(a/h_{t}) \to 0$ and the Toeplitz lemma. This completes the proof of the proposition.
\end{proof}

The next step in the proof of Theorem \ref{bigtheo} is to control the term $M_t(\bx)$ of identity (\ref{M+M}). To reach this goal, we first need a lemma. (In the sequel, the letter $C$ denotes a generic constant whose value may change from line to line.)
\begin{lem}
\label{l11}
Assume that Assumptions {\bf 1-5} are satisfied. Assume, in addition, that $|Y|\leq \gamma$, that (\ref{plus+}) is satisfied, and that
\begin{equation*}
\sup_{t,\bx,\bz} \varepsilon^i_t K_t(\bx,\bz)\leq1,\quad \mbox{for all $i \in \{1, \hdots, M\}$}.
\end{equation*}
Let $\rho\in(0,1)$ be the constant of Lemma \ref{lemmetsitsi}. Then, for all $i\in \{1, \hdots, M\}$ and all $t\geq 1$,
$$\sup_{\bx \in \mathbb R^d}\left|r^i_t(\bx)-r^{\star}_t(\bx)\right| \leq C \xi_t,$$
where $C\geq 0$ is a universal constant independent of $i$, and
$$\xi_t=\sum_{\tau=0}^{t-1} \frac{\rho^{t-\tau}}{(\tau+1)h_{\tau+1}^d}.$$
\end{lem}
\begin{proof}[Proof of Lemma \ref{l11}]
Observe to start with that, for all $i \in \{1, \hdots, M\}$ and all $t\geq 1$,
\begin{equation}
\label{+++}
\sup_{\bx \in \mathbb R^d}\left |H\left(\bZ^i_{t+1},r^i_{t}(\bx)\right)\right| \leq 2\gamma L(0)h_{t+1}^{-d}.
\end{equation}
Here we used the fact that $|Y|$, the $|r_t^i(\bx)|$'s and $|r_t^{\star}(\bx)|$ are bounded by $\gamma$---see Lemma \ref{tKtlemma}.
Now, according to identity (\ref{301213}) and (\ref{301213-2}), we may write
\begin{align*}
\left|r^i_t(\bx)-r^{\star}_t(\bx)\right| & = \left | \sum_{j=1}^M \left ( \phi^{ij}(t,0)-\phi^j_0\right)Y^j_1 + \sum_{\tau=1}^{t-1}\sum_{j=1}^M \left ( \phi^{ij}(t,\tau)-\phi^j_{\tau}\right)s^j_{\tau}\right|\\
& \leq \sum_{j=1}^M\left | \phi^{ij}(t,0)-\phi^j_0\right| \left |Y^j_1\right| + \sum_{\tau=1}^{t-1}\sum_{j=1}^M\left | \phi^{ij}(t,\tau)-\phi^j_{\tau}\right| \left | s^j_{\tau}\right|\\
& \leq (M.\gamma.A)  \rho^{t} + A \sum_{\tau=1}^{t-1}\sum_{j=1}^M \rho^{t-\tau} | s^j_{\tau}|,
\end{align*} 
where $A$ and $\rho$ are the constants of Lemma \ref{lemmetsitsi}. Thus,
\begin{align*}
\left|r^i_t(\bx)-r^{\star}_t(\bx)\right| & \leq (M.\gamma.A)  \rho^{t} +A \sum_{\tau=1}^{t-1}\sum_{j=1}^M \rho^{t-\tau} \mathbf 1_{[\tau \in T^j]}\varepsilon^j_{\tau+1} \left | H\left(\bZ^j_{\tau+1},r^j_{\tau}(\bx)\right)\right|\\
& \leq (M.\gamma.A)  \rho^{t} + (2\gamma L(0).M.A.C_2) \sum_{\tau=1}^{t-1} \frac{\rho^{t-\tau}}{(\tau+1)h_{\tau+1}^d}\\
& \quad (\mbox{by inequality (\ref{+++})}).
\end{align*}
As desired, we conclude that, for some constant $C\geq 0$ independent of $i$,
$$\sup_{\bx \in \mathbb R^d}\left|r^i_t(\bx)-r^{\star}_t(\bx)\right| \leq C \sum_{\tau=0}^{t-1} \frac{\rho^{t-\tau}}{(\tau+1)h_{\tau+1}^d}.$$
\end{proof}

In accordance with our proof plan, the next proposition ensures that the (random) series $(M_t(\bx))_{t\geq 2}$ defined in (\ref{MT}) vanishes in a $L^2$ sense as $t\to \infty$.
\begin{pro}
\label{l12}
Assume that the assumptions of Theorem \ref{bigtheo} are satisfied. Then, provided $(th^d_t)_{t \geq 1}$ is nondecreasing and $\sum_{t \geq 1} \frac{1}{t^2h_t^{2d}}<\infty$, one has
$$\mathbb E \int_{\mathbb R^d} M^2_t(\bx)\mu(\emph{d}\bx) \to 0 \quad\mbox{as } t \to \infty.$$
\end{pro}
\begin{proof}[Proof of Proposition \ref{l12}]
Clearly, for all $i \in \{1, \hdots, M\}$ and all $t\geq 1$,
$$
\left |H\left(\bZ^i_{t+1},r^{i}_{t}(\bx)\right)-H\left(\bZ^i_{t+1},r^{\star}_{t}(\bx)\right)\right| \leq  L(0)h_{t+1}^{-d} \sup_{\bx \in \mathbb R^d} \left |r_t^i(\bx)-r^{\star}_t(\bx)\right|.
$$
On the other hand, for all $\bx \in \mathbb R^d$ and all $t\geq 2$, 
\begin{align*}
M_t(\bx) &\leq \sum_{\tau=2}^t\Delta_{\tau}(\bx)\\
& \leq (L(0).C_2)\sum_{\tau =2} ^t(\tau h^d_{\tau})^{-1}\sum_{j=1}^M\sup_{\bx \in \mathbb R^d} \left |r_{\tau-1}^j(\bx)-r^{\star}_{\tau-1}(\bx)\right|.
\end{align*}
Thus, according to Lemma \ref{l11}, we deduce that for some constant $C\geq 0$,
$$
M_t(\bx)  \leq C\sum_{\tau =2} ^t\frac{\xi_{\tau-1}}{\tau h^d_{\tau}},$$
where   
$$\xi_t=\sum_{\tau=0}^{t-1} \frac{\rho^{t-\tau}}{(\tau+1)h_{\tau+1}^d}.$$
By applying technical Lemma \ref{lemmetechnique} at the end of the section, we conclude that there exists a nonnegative universal constant $C'$ such that $\sup_{t,\bx}M_t(\bx)\leq C'$. Therefore, invoking the Lebesgue dominated convergence theorem, we see that it is enough to prove that, for $\mu$-almost all $\bx \in \mathbb R^d$, $\mathbb E M_t(\bx)\to 0$ as $t\to \infty$.
\medskip

To this aim, let
$$p_t(\bx)=\int_{\mathbb R^d}K_t(\bx,\bz)\mu(\mbox{d}\bz),$$
and recall---by condition (\ref{25.4})---that for $\mu$-almost all $\bx$,
there exists $p(\bx)>0$ and a  large enough time $t_0(\bx)$ such that, for all $t\geq t_0(\bx)$, $p_t(\bx) \geq p(\bx)$. Using similar arguments as in the first part of the proof, we may write
$$M_t(\bx)\leq C\sum_{\tau =2} ^t\bigg[\frac{\xi_{\tau-1}}{\tau h^d_{\tau}}\prod_{\ell=\tau+1}^t \Big(1-\sum_{j=1}^M\mathbf 1_{[\ell-1 \in T^j]}\phi^j_{\ell-1}\varepsilon^j_{\ell}K_{\ell}(\bx,\bX^j_{\ell})\Big)\bigg].$$
Consequently, for $\mu$-almost all $\bx$ and all $t$ large enough (where the ``large enough'' depends upon $\bx$),
\begin{align*}
\mathbb E M_t(\bx)&\leq C\sum_{\tau=2}^{t_0(\bx)-2}\frac{\xi_{\tau-1}}{\tau h^d_{\tau}}\exp \left [-C_1 \eta \sum_{\ell=t_0(\bx)}^t\frac{p(\bx)}{\ell}\right]\\
& + \quad C\sum_{\tau=t_0(\bx)-1}^t\frac{\xi_{\tau-1}}{\tau h^d_{\tau}}\exp \left [-C_1 \eta \sum_{\ell=\tau+1}^t\frac{p(\bx)}{\ell}\right].
\end{align*}
The first term can be made arbitrarily small as $t\to \infty$. To control the second term, recall that
$$\sum_{\ell=1}^t \frac{1}{\ell}=\ln t+ \gamma + \mbox{o}(1) \quad \mbox{as } \ell \to \infty$$
(where $\gamma$ is the Euler-Mascheroni constant). Hence, for some positive constants $\alpha$ and $C$ (both depending upon $\bx$),
$$\sum_{\tau=t_0(\bx)-1}^t\frac{\xi_{\tau-1}}{\tau h^d_{\tau}}\exp \left [-C_1 \eta \sum_{\ell=\tau+1}^t\frac{p(\bx)}{\ell}\right] \leq \frac{C}{t^{\alpha}}\sum_{\tau=2}^t\frac{\tau^{\alpha}\xi_{\tau-1}}{\tau h^d_{\tau}}.$$
According to Lemma \ref{lemmetechnique}, the upper bound tends to zero as $t \to \infty$. 
\end{proof}

We are now ready to finalize the proof of Theorem \ref{bigtheo}. For each $i \in \{1, \hdots, M\}$, we write
\begin{align*}
&\mathbb E\left[\int_{\mathbb R^d}\left|r_t^i(\bx)-r(\bx)\right|^2\mu(\mbox{d}\bx)\right] \\
& \quad \leq 2\mathbb E\left [\sup_{\bx \in \mathbb R^d} \left|r_t^i(\bx)-r^{\star}(\bx)\right|^2\right]+ 2\mathbb E\left[\int_{\mathbb R^d}\left|r_t^{\star}(\bx)-r(\bx)\right|^2\mu(\mbox{d}\bx)\right] \\
& \quad \leq C \xi^2_t + 2\mathbb E\left[\int_{\mathbb R^d}\left|r_t^{\star}(\bx)-r(\bx)\right|^2\mu(\mbox{d}\bx)\right] \\
& \qquad (\mbox{by Lemma \ref{l11}}).
\end{align*}
The first term on the right-hand side tends to zero by technical Lemma \ref{lemmetechnique}. To prove that the second one vanishes as well, just note that
\begin{align*}
&\mathbb E\left[\int_{\mathbb R^d}\left|r_t^{\star}(\bx)-r(\bx)\right|^2\mu(\mbox{d}\bx)\right] \\
& \quad \leq 
2\mathbb E\left[\int_{\mathbb R^d}\left|m_t^{\star}(\bx)-r(\bx)\right|^2\mu(\mbox{d}\bx)\right] + 2\int_{\mathbb R^d} \mathbb E M^2_t(\bx)\mu(\mbox{d}\bx)
\end{align*}
and apply Proposition \ref{HZ} and Proposition \ref{l12}.
\begin{lem} [A technical lemma]
\label{lemmetechnique}
Let $\rho \in (0,1)$ and let
$$\xi_t=\sum_{\tau=0}^{t-1} \frac{\rho^{t-\tau}}{(\tau+1)h_{\tau+1}^d}, \quad \mbox{for all } t \geq 1.$$
If  $th_t^d\to \infty$, then $\xi_t \to 0$ as $t\to \infty$. If, in addition, $(th_t^d)_{t \geq 1}$ is nondecreasing and $\sum_{t \geq 1} \frac{1}{t^2h_t^{2d}}<\infty$, then
$$\sum_{\tau\geq 2} \frac{\xi_{\tau-1}}{\tau h^d_\tau}<\infty \quad \mbox{and} \quad \frac{1}{t^\alpha}\sum_{\tau=2}^t \frac{\tau^{\alpha}\xi_{\tau-1}}{\tau h_\tau^d}\to 0 \quad \mbox{as } t \to \infty,$$
for all $\alpha>0$.
\end{lem}
\begin{proof}[Proof of Lemma \ref{lemmetechnique}]
The first statement is an immediate consequence of Toeplitz lemma.  To prove the second point, fix $t\geq 2$ and note that
\begin{align*}
\sum_{\tau =2} ^t\frac{\xi_{\tau-1}}{\tau h^d_\tau}& =\sum_{\tau=2}^t \frac{1}{\tau h^d_{\tau}}\sum_{\ell=0}^{\tau-2} \frac{\rho^{\tau-1-\ell}}{(\ell+1)h^d_{\ell+1}}\\
& = \sum_{\tau=0}^{t-2} \frac{1}{(\tau+1)h^d_{\tau+1}}\sum_{\ell=\tau+2}^t \frac{\rho^{\ell-\tau-1}}{\ell h^d_{\ell}}\\
& \leq \sum_{\tau=0}^{t-2} \frac{1}{(\tau+1)h^d_{\tau+1}(\tau+2)h^d_{\tau+2}}\sum_{\ell=\tau+2}^t \rho^{\ell-\tau-1}\\
& \leq \frac{\rho}{1-\rho}\sum_{\tau \geq 1}\frac{1}{\tau^2h^{2d}_{\tau}}<\infty.
\end{align*}
Similarly, the third statement follows by writing
\begin{align*}
\frac{1}{t^\alpha}\sum_{\tau =2} ^t\frac{\tau^{\alpha}\xi_{\tau-1}}{\tau h^d_\tau}& =\frac{1}{t^\alpha}\sum_{\tau=2}^t \frac{\tau^{\alpha}}{\tau h^d_{\tau}}\sum_{\ell=0}^{\tau-2} \frac{\rho^{\tau-1-\ell}}{(\ell+1)h^d_{\ell+1}}\\
& = \frac{1}{t^\alpha}\sum_{\tau=0}^{t-2} \frac{1}{(\tau+1)h^d_{\tau+1}}\sum_{\ell=\tau+2}^t \frac{\ell^{\alpha}\rho^{\ell-\tau-1}}{\ell h^d_{\ell}}\\
& \leq \frac{1}{t^\alpha}\sum_{\tau=0}^{t-2} \frac{(\tau+2)^{\alpha}}{(\tau+1)h^d_{\tau+1}(\tau+2)h^d_{\tau+2}}\sum_{k=1}^{t-\tau-1}\left(\frac{k+\tau+1}{\tau+2}\right)^{\alpha}\rho^k\\
&  \leq \frac{1}{t^\alpha}\sum_{\tau=0}^{t-2} \frac{(\tau+2)^{\alpha}}{(\tau+1)h^d_{\tau+1}(\tau+2)h^d_{\tau+2}}\sum_{k=1}^{t-\tau-1}k^{\alpha}\rho^k\\
& \leq \frac{C}{t^\alpha}\sum_{\tau=1}^{t-1} \frac{(\tau+1)^{\alpha}}{\tau^2h_{\tau}^{2d}},
\end{align*}
for some positive constant $C$. Now, fix $\varepsilon >0$ and choose $T_0$ large enough so that
$$\sum_{\tau \geq T_0} \frac{1}{\tau^2 h_{\tau}^{2d}}\leq \varepsilon.$$
Then, for all $t$ large enough,
\begin{align*}
\frac{1}{t^\alpha}\sum_{\tau=1}^{t-1} \frac{(\tau+1)^{\alpha}}{\tau^2h_{\tau}^{2d}}& =\frac{1}{t^\alpha}\sum_{\tau=1}^{T_0-1} \frac{(\tau+1)^{\alpha}}{\tau^2h_{\tau}^{2d}}+\frac{1}{t^\alpha}\sum_{\tau=T_0}^{t-1} \frac{(\tau+1)^{\alpha}}{\tau^2h_{\tau}^{2d}}\\
& \leq \frac{1}{t^\alpha}\sum_{\tau=1}^{T_0-1} \frac{(\tau+1)^{\alpha}}{\tau^2h_{\tau}^{2d}}+ \varepsilon.
\end{align*}
The conclusion follows by letting $t$ grow to infinity.
\end{proof}
\bibliography{biblio-biauzenine}

\begin{thebibliography}{22}
\providecommand{\natexlab}[1]{#1}
\providecommand{\url}[1]{\texttt{#1}}
\expandafter\ifx\csname urlstyle\endcsname\relax
  \providecommand{\doi}[1]{doi: #1}\else
  \providecommand{\doi}{doi: \begingroup \urlstyle{rm}\Url}\fi

\bibitem[Bertsekas and Tsitsiklis(1997)]{BeTs89}
D.P. Bertsekas and J.N. Tsitsiklis.
\newblock \emph{Parallel and Distributed Computation: Numerical Methods}.
\newblock Athena Scientific, Nashua, 1997.

\bibitem[Bianchi et~al.(2011{\natexlab{a}})Bianchi, Fort, Hachem, and
  Jakubowicz]{BiFoHaJa11}
P.~Bianchi, G.~Fort, W.~Hachem, and J.~Jakubowicz.
\newblock Convergence of a distributed parameter estimator for sensor networks
  with local averaging of the estimates.
\newblock In \emph{Proceedings of the 36th IEEE International Conference on
  Acoustics, Speech and Signal Processing}, 2011{\natexlab{a}}.

\bibitem[Bianchi et~al.(2011{\natexlab{b}})Bianchi, Fort, Hachem, and
  Jakubowicz]{BiFoHaJa11-2}
P.~Bianchi, G.~Fort, W.~Hachem, and J.~Jakubowicz.
\newblock Performance analysis of a distributed {R}obbins-{M}onro algorithm for
  sensor networks.
\newblock In \emph{Proceedings of the 19th European Signal Processing
  Conference}, 2011{\natexlab{b}}.

\bibitem[Bianchi et~al.(2013)Bianchi, Cl{\'e}men\c{c}on, Jakubowicz, and
  Adel]{BiClJaMo13}
P.~Bianchi, S.~Cl{\'e}men\c{c}on, J.~Jakubowicz, and G.~Morra Adel.
\newblock On-line learning gossip algorithm in multi-agent systems with local
  decision rules.
\newblock In \emph{Proceedings of the 2013 IEEE International Conference on Big
  Data}, 2013.

\bibitem[Blondel et~al.(2005)Blondel, Hendrickx, Olshevsky, and
  Tsitsiklis]{BlHeOlTs05}
V.D. Blondel, J.M. Hendrickx, A.~Olshevsky, and J.N. Tsitsiklis.
\newblock Convergent in multiagent coordination, consensus, and flocking.
\newblock In \emph{Proceedings of the Joint 44th IEEE Conference on Decision
  and Control and European Control Conference}, 2005.

\bibitem[Boyd et~al.(2006)Boyd, Ghosh, Prabhakar, and Shah]{BoGhPrSh06}
S.~Boyd, A.~Ghosh, B.~Prabhakar, and D.~Shah.
\newblock Randomized gossip algorithms.
\newblock \emph{IEEE Transactions on Information Theory}, 52:\penalty0
  2508--2530, 2006.

\bibitem[{Gy\"orfi}(1981)]{Gy81}
L.~{Gy\"orfi}.
\newblock Recent results on nonparametric regression estimate and multiple
  classification.
\newblock \emph{Problems of Control and Information Theory}, 10:\penalty0
  43--52, 1981.

\bibitem[{Gy\"orfi} and Walk(1996)]{GyWa96}
L.~{Gy\"orfi} and H.~Walk.
\newblock On the strong universal consistency of a series type regression
  estimate.
\newblock \emph{Mathematical Methods of Statistics}, 5:\penalty0 332--342,
  1996.

\bibitem[{Gy\"orfi} and Walk(1997)]{GyWa97}
L.~{Gy\"orfi} and H.~Walk.
\newblock On the strong universal consistency of a recursive regression
  estimate by {P\'al} {R\'ev\'esz}.
\newblock \emph{Statistics {\&} Probability Letters}, 31:\penalty0 177--183,
  1997.

\bibitem[{Gy\"orfi} et~al.(2002){Gy\"orfi}, Kohler, {Krzy\.zak}, and
  Walk]{GyKoKr02}
L.~{Gy\"orfi}, M.~Kohler, A.~{Krzy\.zak}, and H.~Walk.
\newblock \emph{A Distribution-Free Theory of Nonparametric Regression}.
\newblock Springer, New York, 2002.

\bibitem[Jordan(2013)]{Jo13}
M.I. Jordan.
\newblock On statistics, computation and scalability.
\newblock \emph{Bernoulli}, 19:\penalty0 1378--1390, 2013.

\bibitem[Kiefer and Wolfowitz(1952)]{KiWo52}
J.~Kiefer and J.~Wolfowitz.
\newblock Stochastic estimation of the maximum of a regression function.
\newblock \emph{The Annals of Mathematical Statistics}, 23:\penalty0 462--466,
  1952.

\bibitem[Mokkadem et~al.(2009)Mokkadem, Pelletier, and Slaoui]{MoPeSl09}
A.~Mokkadem, M.~Pelletier, and Y.~Slaoui.
\newblock Revisiting {R\'ev\'esz} stochastic approximation method for the
  estimation of a regression function.
\newblock \emph{{ALEA}}, 6:\penalty0 63--114, 2009.

\bibitem[Patra(2011)]{Pa11}
B.~Patra.
\newblock Convergence of distributed asynchronous learning vector quantization
  algorithms.
\newblock \emph{Journal of Machine Learning Research}, 12:\penalty0 3431--3466,
  2011.

\bibitem[{R\'ev\'esz}(1973)]{Re73}
P.~{R\'ev\'esz}.
\newblock {Robbins}-{Monro} procedure in a {Hilbert} space and its application
  in the theory of learning processes {I}.
\newblock \emph{Studia Scientiarum Mathematicarum Hungarica}, 8:\penalty0
  391--398, 1973.

\bibitem[{R\'ev\'esz}(1977)]{Re77}
P.~{R\'ev\'esz}.
\newblock How to apply the method of stochastic approximation in the
  non-parametric estimation of a regression function.
\newblock \emph{Series Statistics}, 8:\penalty0 119--126, 1977.

\bibitem[Robbins and Monro(1951)]{RoMo51}
H.~Robbins and S.~Monro.
\newblock A stochastic approximation method.
\newblock \emph{The Annals of Mathematical Statistics}, 22:\penalty0 400--407,
  1951.

\bibitem[Stein(1970)]{St70}
E.M. Stein.
\newblock \emph{Singular Integrals and Differentiability Properties of
  Functions}.
\newblock Princeton University Press, Princeton, 1970.

\bibitem[Stone(1977)]{St77}
C.J. Stone.
\newblock Consistent nonparametric regression (with discussion).
\newblock \emph{The Annals of Statistics}, 5:\penalty0 595--645, 1977.

\bibitem[Tsitsiklis(1984)]{Ts84}
J.N. Tsitsiklis.
\newblock \emph{Problems in decentralized decision making and computation}.
\newblock Ph.D. Thesis, Massachusetts Institute of Technology, Cambridge, 1984.

\bibitem[Tsitsiklis et~al.(1986)Tsitsiklis, Bertsekas, and Athans]{TsBeAt86}
J.N. Tsitsiklis, D.P. Bertsekas, and M.~Athans.
\newblock Distributed asynchronous deterministic and stochastic gradient
  optimization algorithms.
\newblock \emph{{IEEE} Transactions on Automatic Control}, 31:\penalty0
  803--812, 1986.

\bibitem[Walk(2001)]{Wa01}
H.~Walk.
\newblock Strong universal pointwise consistency of recursive regression
  estimates.
\newblock \emph{Annals of the Institute of Statistical Mathematics},
  53:\penalty0 691--707, 2001.

\end{thebibliography}
\end{document}